\documentclass[draft]{amsart}
\usepackage{amsmath}
\usepackage{amsthm}
\usepackage{amssymb}
\newtheorem{theorem}{Theorem}[section]
\newtheorem{lemma}[theorem]{Lemma}
\newtheorem{proposition}[theorem]{Proposition}

\theoremstyle{definition}

\newtheorem{example}[theorem]{Example}
\theoremstyle{remark}
\newtheorem{remark}[theorem]{Remark}

\numberwithin{equation}{section}
\numberwithin{equation}{section}

\begin{document}

\title[The geometric mean errors for Markov-type measures]{Convergence order of the geometric mean errors for Markov-type measures}
\author{Sanguo Zhu}
\address{School of Mathematics and Physics, Jiangsu University of Technology,
Changzhou 213001, China.} \email{sgzhu@jsut.edu.cn}


\begin{abstract}
We study the quantization problem with respect to the geometric mean error for Markov-type measures $\mu$ on a class of fractal sets. Assuming the irreducibility of the corresponding transition matrix $P$, we determine the exact convergence order of the geometric mean errors of $\mu$. In particular, we show that, the quantization dimension of order zero is independent of the initial probability vector when $P$ is irreducible, while this is not true if $P$ is reducible.
\end{abstract}
\maketitle

\section{Introduction}
In this paper, we study the asymptotic geometric mean errors in the quantization for Markov-type measures on a class of fractal sets. We refer to \cite{GL:00,GL:04} for mathematical foundations of quantization theory and \cite{GN:98} for its background in engineering technology.

For every $n\geq 1$, we set $\mathcal{D}_n:=\{\alpha\subset\mathbb{R}^q:1\leq{\rm card}(\alpha)\leq n\}$. Let $\nu$  be a Borel probability measure on $\mathbb{R}^q$. The $n$th quantization error  for $\nu$ of order $r$ is defined by (see \cite{GL:00,GL:04}):
\begin{eqnarray}\label{quanerrordef}
e_{n,r}(\nu):=\left\{\begin{array}{ll}\inf_{\alpha\in\mathcal{D}_n}\big(\int d(x,\alpha)^{r}d\nu(x)\big)^{\frac{1}{r}}
,\;\;\;\;\;\;r>0,\\
\inf_{\alpha\in\mathcal{D}_n}\exp\int\log d(x,\alpha)d\nu(x),\;\;\;r=0.\end{array}\right.
\end{eqnarray}
Here $d(\cdot,\cdot)$ is the metric induced by an arbitrary norm on $\mathbb{R}^q$.
For $r>0$, $e_{n,r}(\nu)$ agrees with the error in the approximation of $\nu$ by discrete probability measures supported on at most $n$ points,
in the sense of $L_r$-metrics \cite{GL:00}.

By \cite[Lemma 3.5]{GL:04}, the quantity $e_{n,0}(\nu)$---also called the $n$th geometric mean
error for $\nu$, equals the limit of $e_{n,r}(\nu)$ as $r$ tends to zero. In this sense, the quantization with respect to the geometric mean error is a limiting case of that in $L_r$-metrics. As one of the main aims of the quantization problem, we are concerned with the asymptotic properties of the quantization errors, including the upper (lower) quantization coefficient (of order $r$) and the upper (lower) quantization dimension (of order $r$).

For $s>0$, we define the $s$-dimensional upper and lower quantization coefficient for $\nu$ of order $r\in[0,\infty)$ by (cf. \cite{GL:00,PK:01})
\begin{eqnarray*}
\overline{Q}_r^s(\nu):=\limsup_{n\to\infty}n^{\frac{1}{s}}e_{n,r}(\nu),\;\;
\underline{Q}_r^s(\nu):=\liminf_{n\to\infty}n^{\frac{1}{s}}e_{n,r}(\nu).
\end{eqnarray*}
By \cite{GL:00,PK:01}, the upper (lower) quantization dimension $\overline{D}_r(\nu)$ ($\underline{D}_r(\nu)$) of order $r$, as defined below, is exactly the critical point at which the upper (lower) quantization coefficient jumps from zero to infinity:
\begin{eqnarray*}
\overline{D}_r(\nu):=\limsup_{n\to\infty}\frac{\log n}{-\log e_{n,r}(\nu)},\;\underline{D}_r(\nu):=\liminf_{n\to\infty}\frac{\log n}{-\log e_{n,r}(\nu)}.
\end{eqnarray*}
If $\overline{D}_r(\nu)=\underline{D}_r(\nu)$, the common value is denoted by $D_r(\nu)$ and called the quantization dimension for $\nu$ of order $r$.

Compared with the upper (lower) quantization dimension of order $r$, the upper (lower) quantization coefficient of order $r$ provides us with more accurate information on the asymptotics of the geometric mean errors; accordingly, much more effort is required to examine the finiteness and positivity of the latter.
\begin{remark}{\rm
The upper (lower) quantization dimension of $\nu$ of order zero is closely connected with the upper (lower) local dimension (cf. \cite{Fal:97}) as defined by
\[
\underline{\dim}_{\rm loc}\nu(x):=\liminf_{\epsilon\to 0}\frac{\log\nu(B_\epsilon(x))}{\log\epsilon},\;\overline{\dim}_{\rm loc}\nu(x):=\limsup_{\epsilon\to 0}\frac{\log\nu(B_\epsilon(x))}{\log\epsilon}.
\]
Here $B_\epsilon(x)$ denotes the closed ball of radius $\epsilon$ which is centered at a point $x\in\mathbb{R}^q$. As we showed in \cite{zhu:12}, if the upper and lower local dimension are both equal to $s$ for $\nu$-a.e. $x$, then $D_0(\nu)$ exists and equals $s$.
}\end{remark}
Next, let us recall a result of Graf and Luschgy. Let $(f_i)_{i=1}^N$ be a family of contractive similitudes on $\mathbb{R}^q$ with contraction ratios $(s_i)_{i=1}^N$. By \cite{Hut:81}, there exists a unique non-empty compact set $K$ satisfying
\begin{eqnarray}\label{g30}
K=f_1(K)\cup f_2(K)\cdots\cup f_N(K).
\end{eqnarray}
This set is called the self-similar set associated with $(f_i)_{i=1}^N$. Also, By \cite{Hut:81}, for a probability vector $(q_i)_{i=1}^N$, there exists a unique
Borel probability measure $\nu$ satisfying $\nu=\sum_{i=1}^Nq_i\nu\circ f_i^{-1}$. This measure is called the self-similar measure associated with $(f_i)_{i=1}^N$ and $(q_i)_{i=1}^N$.

We say that $(f_i)_{i=1}^N$ satisfies the open set condition (OSC), if there exists a non-empty bounded open set $U$ such that
$f_i(U)\cap f_j(U)=\emptyset$ for all $1\leq i\neq j\leq N$ and $f_i(U)\subset U$ for all $1\leq i\leq N$.
Let $k_r,r\geq 0$, be given by
\[
k_0:=\frac{\sum_{i=1}^Nq_i\log q_i}{\sum_{i=1}^Nq_i\log s_i};\;\;\sum_{i=1}^N(q_is_i^r)^{\frac{k_r}{k_r+r}}=1,\;r>0.
\]
Assume that $(f_i)_{i=1}^N$ satisfies the OSC. Let $\nu$ be the self-similar measure associated with $(f_i)_{i=1}^N$ and $(q_i)_{i=1}^N$. Graf and Luschgy proved \cite{GL:01,GL:04} that
$$
D_r(\nu)=k_r;\;\;0<\underline{Q}_r^{k_r}(\nu)\leq\overline{Q}_r^{k_r}(\nu)<\infty.
$$

In the present paper, we study the finiteness and positivity of the upper and lower quantization coefficient of order zero for Markov-type measures. Some results on the quantization for such measures in $L_r$-metrics (with $r>0$) are contained in \cite{LJL:01}.  Recently, a complete treatment in this direction is given in \cite{kz:14}, where the corresponding transition matrix is allowed to be reducible. Next, let us recall some definitions; we refer to \cite{BED:89,EM:92,MW:88} for more details.

Let $P=(p_{ij})_{N\times N}$ be a row-stochastic matrix, i.e.,
$p_{ij}\geq 0,1\leq i,j\leq N$; and $\sum_{j=1}^Np_{ij}=1,1\leq i\leq N$. It is easy to see that, $1$ is an eigenvalue of $P$ of largest absolute value (cf. \cite[Theorem 8.1.22]{HJ:85}). When $P$ is irreducible, by the Perron-Frobenius theorem (cf. \cite[Theorem 8.4.4]{HJ:85}), there exists a unique normalized positive left (row) eigenvector $v=(v_1,\ldots,v_N)$ of $P$ with respect to the eigenvalue $1$. We will need the following notations:
\begin{eqnarray*}
&&\theta:={\rm empty\;word},\;\;G_0:=\{\theta\},\;G_1:=\{1,\ldots N\};\\&&G_k:=\{\sigma\in G_1^k:p_{\sigma_1\sigma_2}\cdots p_{\sigma_{k-1}\sigma_k}>0\},\;k\geq 2;\\
&&G^*:=\bigcup_{k\geq 0} G_k,\;G_\infty:=\{\sigma\in G_1^{\mathbb{N}}:p_{\sigma_h\sigma_{h+1}}>0\;\;{\rm for\; all}\;\;h\geq 1\}.
\end{eqnarray*}

We define $|\sigma|:=k$ for $\sigma\in G_k$ and $|\sigma|:=\infty$ for $\sigma\in G_1^{\mathbb{N}}$. For every $\sigma\in G^*$ with $|\sigma|\geq k$ or
$\sigma\in G_\infty$, we write $\sigma|_k:=(\sigma_1,\ldots,\sigma_k)$. For $\sigma\in G^*$ and $\omega\in G^*\cup G_\infty$ with $(\sigma_{|\sigma|}, \omega_1)\in G_2$, then we set
$$
\sigma\ast\omega=(\sigma_1,\sigma_2,\ldots,\sigma_{|\sigma|},\omega_1,\ldots,\omega_{|\omega|}).
$$

Let $J_i,i\in G_1$, be non-empty compact subsets of $\mathbb{R}^q$ with $J_i=\overline{\rm int(J_i)}$ for all $i\in G_1$, where $\overline{B}$ and ${\rm int}(B)$ respectively denote the closure and interior in $\mathbb{R}^q$ of a set $B\subset\mathbb{R}^q$. We call these sets cylinders of order one. For each $i\in G_1$, let $J_{ij},\;(i,j)\in G_2$, be non-overlapping subsets of $J_i$ such that $J_{ij}$ is geometrically similar to $J_j$
and $|J_{ij}|/|J_j|=c_{ij}$, where $|A|$ denotes the diameter of a set $A$ and $c_{ij}\in(0,1)$. We call these sets cylinders of order two. Assume that cylinders of order $k$ are determined. Let $J_{\sigma\ast i_{k+1}},\sigma\ast i_{k+1}\in G_{k+1}$,
be non-overlapping subsets of $J_\sigma$ such that $J_{\sigma\ast i_{k+1}}$ is geometrically similar to $J_{i_{k+1}}$. Hence, by induction, cylinders of order $k$
are determined for all $k\geq 1$. Then, we get a Mauldin-Williams fractal set $E$ (cf. \cite{BED:89,MW:88}):
\begin{equation*}
E:=\bigcap_{k\geq 1}\bigcup_{\sigma\in G_k}J_\sigma.
\end{equation*}
The set $E$ need not be a self-similar set, and in general, $E$ does not enjoy the nice invariance property as in (\ref{g30}). This will cause much difficulty in the study of the geometric mean error. For this reason, we assume the following separation property for $E$: there exists some constant $0<t<1$ such that for every $\sigma\in G^*$ and $j_l$ with $(\sigma_{|\sigma|},j_l)\in G_2,l=1,2$,
\begin{equation}\label{g4}
d(J_{\sigma\ast j_1},J_{\sigma\ast j_2})\geq t\max\{|J_{\sigma\ast j_1}|,|J_{\sigma\ast j_2}|)\}.
\end{equation}

Let $(q_i)_{i=1}^N$ be an arbitrary probability vector with $q_i>0$ for all $i\in G_1$. By Kolmogorov consistency theorem, there exists a unique Markov-type measure $\widetilde{\mu}$ on $G_\infty$ (cf. \cite{PW:82}) such that, for every $k\geq 1$ and $\sigma=(\sigma_1\ldots \sigma_k)\in G_k$,
\[
\widetilde{\mu}([\sigma])=q_{\sigma_1}p_{\sigma_1\sigma_2}\cdots p_{\sigma_{k-1}\sigma_k},
\]
where $[\sigma]:=\{\sigma\ast\omega:\omega\in G_\infty,\;(\sigma_{|\sigma|},\omega_1)\in G_2\}$.
With the assumption (\ref{g4}), we have the following bijection $g:G_\infty\to E$:
\[
g(\sigma):=\bigcap_{k\geq 1}J_{\sigma|_k},\;\;\sigma\in G_\infty.
\]
Let $\mu:=\widetilde{\mu}\circ g^{-1}$. Then $\mu$ is a Markov-type measure on $E$ satisfying
\begin{eqnarray}\label{markovmeasure}
\mu(J_\sigma)= q_{\sigma_1}p_{\sigma_1\sigma_2}\cdots p_{\sigma_{k-1}\sigma_k}\;\; {\rm for}\;\; \sigma=(\sigma_1\ldots \sigma_k)\in G_k,\;k\geq 1.
\end{eqnarray}
For each $\sigma\in G^*\setminus{\theta}$, the way in which $\mu$ distributes its measure among the sub-cylinders of $J_\sigma$ depends on the last entry of $\sigma$. This is different from that of the measures as considered in \cite{Zhu:13}.
From now on, we assume
\begin{equation}\label{cardpij>0}
{\rm card}\big(\{j\in G_1:(i,j)\in G_2\}\big)\geq 2\;\; {\rm for\; all}\;\; i\in G_1.
\end{equation}
Under the assumption (\ref{cardpij>0}), for each $\sigma\in G^*$, the cylinder $J_\sigma$ has at least two sub-cylinders of order $|\sigma|+1$; in addition, we have that $\max_{(i,j)\in G_2}p_{ij}<1$.

As the main result of the present paper, we will determine the exact convergence order of the geometric mean error for $\mu$. That is,
\begin{theorem}\label{mthm1}
Assume that (\ref{g4}) and (\ref{cardpij>0}) are satisfied. Let $\mu$ be as given in (\ref{markovmeasure}). Assume that
the transition matrix $P$ is irreducible. Then we have, $D_0(\mu)=s_0$ and $0<\underline{Q}_0^{s_0}(\mu)\leq\overline{Q}_0^{s_0}(\mu)<\infty$, where
\begin{eqnarray}\label{s0}
s_0:=\frac{\sum_{i=1}^N v_i\sum_{j:(i,j)\in G_2}p_{ij}\log p_{ij}}{\sum_{i=1}^N v_i\sum_{j:(i,j)\in G_2}p_{ij}\log c_{ij}}.
\end{eqnarray}
and $(v_i)_{i=1}^N$ is the normalized positive left eigenvector of $P$ with respect to $1$.
\end{theorem}

By Theorem  \ref{mthm1}, when the transition matrix $P$ is irreducible, $D_0(\mu)$ is independent of the initial probability vector $(q_i)_{i=1}^N$. In this case, according to \cite{DGSH:91,EM:92}, $D_0(\mu)$ coincides with the Hausdorff dimension of $\mu$. At the end of the paper, we will give an example to show that, $(q_i)_{i=1}^N$ usually plays a role in the quantization with respect to the geometric mean error for $\mu$ when the transition matrix is reducible.

\section{A characterization of the geometric mean error}

For every $k\geq 2$ and $\sigma=(\sigma_1,\ldots,\sigma_k)\in G_k$, we write
\begin{eqnarray*}
\sigma^-:=\sigma|_{k-1};\;\;p_\sigma:=p_{\sigma_1\sigma_2}\cdots p_{\sigma_{k-1}\sigma_k},\;c_\sigma:=c_{\sigma_1\sigma_2}\cdots c_{\sigma_{k-1}\sigma_k}.
\end{eqnarray*}
If $|\sigma|=1$, we define $p_\sigma=c_\sigma=1$ and $\sigma^-=\theta$. If $\sigma,\omega\in G^*$ satisfy $|\sigma|\leq|\omega|$ and $\sigma=\omega|_{|\sigma|}$, then write $\sigma\prec\omega$. Set
\begin{eqnarray}\label{lambdaj}
&&\underline{p}:=\min_{(i,j)\in G_2}p_{ij},\;\underline{c}:=\min_{(i,j)\in G_2}c_{ij},\;
\overline{p}:=\max_{(i,j)\in G_2}p_{ij},\;\overline{c}:=\max_{(i,j)\in G_2}c_{ij};\nonumber\\
&&\Lambda_j:=\{\sigma\in G^*:p_{\sigma^-}\geq \underline{p}^j>p_\sigma\},\;\;
\psi_j:={\rm card}(\Lambda_j);\label{g9}\\&&k_{1j}:=\min_{\sigma\in\Lambda_j}|\sigma|,\;k_{2j}:=\max_{\sigma\in\Lambda_j}|\sigma|;
\nonumber\\ &&\underline{P}_0^s(\mu):=\liminf_{j\to\infty}\psi_j^{\frac{1}{s}}e_{\psi_j}(\mu),\;
\overline{P}_0^s(\mu):=\limsup_{j\to\infty}\psi_j^{\frac{1}{s}}e_{\psi_j}(\mu),\;s>0.\nonumber
\end{eqnarray}
Without loss of generality, we assume that $|J_i|=1$ for all $i\in G_1$. Thus,
\[
|J_\sigma|=c_\sigma,\;\;\sigma\in G_k,\;k\geq 1.
\]
\begin{lemma}\label{g7}
(i) There exist constants $A_1,A_2>0$ such that
\[
A_1j\leq k_{1j}\leq k_{2j}\leq A_2j.
\]
(ii) $\underline{Q}_0^s(\mu)>0$ iff $\underline{P}_0^s(\mu)>0$ and $\overline{Q}_0^s(\mu)<\infty$ iff $\overline{P}_0^s(\mu)<\infty$.
\end{lemma}
\begin{proof}
Let $N_1:=\min\{h\geq 1:\overline{p}^h<\underline{p}\}$.
Let $\sigma^{(l)}\in G_{k_{lj}}\cap\Lambda_j,l=1,2$. Then
\begin{eqnarray*}
\underline{p}^{k_{1j}}\leq p_{\sigma^{(1)}}<\underline{p}^j,\;\;\overline{p}^{k_{2j}-1}\geq
p_{\sigma^{(2)}}\geq \underline{p}^{j+1}\geq\overline{p}^{N_1(j+1)}.
\end{eqnarray*}
It follows that $j\leq k_{1j}\leq k_{2j}\leq N_1(j+1)+1\leq 3N_1j$, for all $j\geq 1$. Hence (i) follows by setting $A_1:=1$ and $A_2:=3N_1$. As in \cite{Zhu:13},
to see (ii), it suffices to show that for some constant $N_2>0$ such that $\psi_j\leq\psi_{j+1}\leq N_2\psi_j$. In fact, the first inequality is clear;
to see the second, we note that, for every $\sigma\in\Lambda_j$ and every $\omega\in G_{N_1+1}$ with $(\sigma_{|\sigma|}, \omega_1)\in G_2$, we have
$p_{\sigma\ast\omega}<\underline{p}^j \overline{p}^{N_1}<\underline{p}^{j+1}$.
This implies that $\psi_{j+1}\leq N^{N_1+1}\psi_j$. The lemma follows.
\end{proof}
\subsection{Push-forward and pull-back measures}
For each $\sigma\in G^*\setminus\{\theta\}$, we take an arbitrary contracting similitude $f_\sigma$ on $\mathbb{R}^q$ of contraction ratio $c_\sigma$ and define $\nu_\sigma:=\mu(\cdot|J_\sigma)\circ f_\sigma$. Then, since $f_\sigma$ is a Borel bijection, $\nu_\sigma$ is a measure supported on $K(\sigma)= f_\sigma^{-1}(J_\sigma)$ satisfying
\begin{eqnarray}\label{z17}
\mu(\cdot|J_\sigma)=\nu_\sigma\circ f_\sigma^{-1},\;\;\sigma\in G^*\setminus\{\theta\}.
\end{eqnarray}
For a finite set $\alpha\subset\mathbb{R}^q$ of cardinality $L$, by (\ref{z17}), we have
\begin{eqnarray}\label{g6}
&&\int_{J_\sigma}\log d(x,\alpha)d\mu(x)=\mu(J_\sigma)\int\log d(x,\alpha)d\nu_\sigma\circ f_\sigma^{-1}(x)\nonumber\\&&=\mu(J_\sigma)\int\log d(f_\sigma(x),\alpha)d\nu_\sigma(x)\geq\mu(J_\sigma)(\log c_\sigma+\hat{e}_L(\nu_\sigma)).
\end{eqnarray}
\begin{lemma}
There exist constants $A_3,A_4>0$, such that
\begin{eqnarray}\label{s8}
\sup_{\sigma\in G^*\setminus\{\theta\}}\sup_{x\in \mathbb{R}^q}\nu_\sigma(B(x,\epsilon))\leq A_3\epsilon^{A_4}\;\;{\rm for\; all}\;\;\epsilon>0.
\end{eqnarray}
\end{lemma}
\begin{proof}
Let $\sigma\in G^*\setminus\{\theta\},x\in K(\sigma)$. By (\ref{g4}), there exists a unique word $\tau_x\in G_\infty$ such that $\sigma\prec\tau_x$ and
$\bigcap_{k\geq 1}J_{\tau_x|_k}=\{x\}$. For every $\epsilon\in(0,\underline{c})$, we set
\begin{eqnarray}\label{z13}
\mathcal{C}(\sigma):=\{\tau\in G^*:\sigma\prec\tau,\;c_\sigma^{-1}c_{\tau^-}\geq\epsilon>c_\sigma^{-1}c_\tau\}.
\end{eqnarray}
For each $i\in G_1$, there exists some $t_i\in (0,1)$ such that $J_i$ contains a ball of radius $t_i|J_i|=t_i$ and is contained in a closed ball of radius $1$. Set $\delta:=\min_{1\leq i\leq N}t_i$. Then, for each $\tau\in\mathcal{C}(\sigma)$, $f_\sigma^{-1}(J_\tau)$ is contained in a ball of radius $\epsilon$ and contains a ball of radius $\delta\underline{c}\epsilon$. By (\ref{g4}), $J_\tau,\tau\in\mathcal{C}(\sigma)$ are pairwise disjoint, so are the sets $f_\sigma^{-1}(J_\tau),\tau\in\mathcal{C}(\sigma)$ by the similarity of $f_\sigma$. Thus, by \cite{Hut:81}, there is a constant $M$ which is independent of $\epsilon$ such that
\[
{\rm card}(\{\tau\in\mathcal{C}(\sigma):B(x,\epsilon)\cap f_\sigma^{-1}(J_\tau)\neq\emptyset\})\leq M.
\]
By (\ref{z13}), $\underline{c}^{|\tau|-|\sigma|}<\epsilon$ for $\tau\in\mathcal{C}(\sigma)$, which implies $|\tau|-|\sigma|\geq\log\epsilon/\log\underline{c}$. So,
\begin{eqnarray*}
\nu_\sigma(B(x,\epsilon))\leq M\overline{p}^{\frac{\log\epsilon}{\log\underline{c}}}=M\epsilon^{\frac{\log\overline{p}}{\log\underline{c}}}.
\end{eqnarray*}
Let $A_4:=\frac{\log\overline{p}}{\log\underline{c}}$. Then by \cite[Lemma 12.3]{GL:00}, there is a constant $A_3>0$, independent of $\sigma$, such that $\nu_\sigma(B(x,\epsilon))\leq A_3\epsilon^{A_4}$ for all $x\in\mathbb{R}^q$. Note that the above arguments holds true for any $\sigma\in G^*\setminus\{\theta\}$. The lemma follows.
\end{proof}

If the infimum in (\ref{quanerrordef}) is attained at some $\alpha$ with $1\leq{\rm card}(\alpha)\leq n$, we call $\alpha$ an $n$-optimal set for $\nu$ of order $r$. The collection of all $n$-optimal sets for $\nu$ of order $r$ is denoted by $C_{n,r}(\nu)$. We simply write $C_n(\nu)$ for $C_{n,0}(\nu)$.
Note that $\nu_\sigma,\sigma\in G^*$, are compactly supported. By Lemma \ref{s8} and \cite[Theorem 2.5]{GL:00}, we conclude that $C_n(\nu_\sigma)$ is non-empty for every $\sigma\in G^*\setminus\{\theta\}$ and $n\geq 1$. Using similar arguments, one can show that $C_n(\mu)$ is non-empty for every $n\geq 1$.

\begin{lemma}\label{g29}(see \cite{GL:04})
Let $\nu$ be a Borel probability measure on $\mathbb{R}^q$ with compact support $K$. Let $\hat{e}_n(\nu):=\log e_{n,0}(\nu)$. Assume that for some constants $d_1,d_2>0$ we have, $\sup_{x\in\mathbb{R}^q}\nu(B(x,\epsilon))\leq d_1\epsilon^{d_2}$. Then, we have
\[
\hat{e}_n(\nu)-\hat{e}_{n+1}(\nu)\leq(n+1)^{-1}\log(3|K|)+d_1^{1/q}qd_2^{-1}(n+1)^{-1/p},\;n\geq 1.
\]
where $p,q$ are real numbers satisfying $p,q>1,p^{-1}+q^{-1}=1$.
\end{lemma}
Let $\underline{q}:=\min_{i\in G_1} q_i$ and $\overline{q}:=\max_{i\in G_1} q_i$. As a consequence of (\ref{s8}) and Lemma \ref{g29}, for given integers $k_1,k_2,k_3\geq 1$, there exists an integer $A_5$ such that, for all $n\geq A_5$, we have
\begin{eqnarray}\label{z8}
\sup_{\sigma\in G^*}\big(\hat{e}_{n-k_1-k_3}(\nu_\sigma)-\hat{e}_{n+k_2}(\nu_\sigma)\big)<\underline{q}\overline{q}^{-1}\underline{p}\log 2.
\end{eqnarray}
\begin{remark}\label{z7} {\rm Using (\ref{s8}) and the proof of Theorem 3.4 of \cite{GL:04}, it is convenient to see, for every $k\geq 1$, there is a $B_k\in\mathbb{R}$ such that $\inf_{\sigma\in G^*}\hat{e}_k(\nu_\sigma)\geq B_k$.
}\end{remark}

\subsection{An estimate of the geometric mean error}
For $\epsilon>0$, let $(A)_\epsilon$ denote the closed $\epsilon$-neighborhood in $\mathbb{R}^q$ of a set $A\subset\mathbb{R}^q$. Let $t$ be the same as in (\ref{g4}). For a finite subset $\alpha$ of $\mathbb{R}^q$ and $\sigma\in G^*$, we write $\alpha_\sigma:=\alpha\cap(J_\sigma)_{4^{-1}tc_\sigma}$ and
$$
L_\sigma:={\rm card}(\alpha_\sigma),\;I_\sigma(\alpha):=\int_{J_\sigma} \log d(x,\alpha)d\mu(x).
$$
By (\ref{markovmeasure}), we have, $\mu(J_\sigma)=q_{\sigma_1}p_\sigma$ for every $\sigma\in G^*\setminus\{\theta\}$. We set
\[
J_\theta:=E;\;\;m_\sigma:=q_{\sigma_1}p_\sigma,\;\sigma\in G^*\setminus\{\theta\}.
\]
\begin{lemma}
There exists a constant $L_1$ which is independent of $j$ such that
\begin{eqnarray}\label{g3}
\sup_{\alpha\in C_{\psi_j}(\mu)}\max_{\sigma\in\Lambda_j}L_\sigma\leq L_1.
\end{eqnarray}
\end{lemma}
\begin{proof}
Since all $\nu_\sigma,\sigma\in G^*$, share the properties in (\ref{z8}) and (\ref{z7}),
it suffices to follow the induction in \cite[Proposition 3.4]{Zhu:10} by using (\ref{z17}).
\end{proof}

\begin{remark}\label{z20}{\rm
For the reader's convenience, let us explain the main idea of the induction in \cite{Zhu:10} by contradiction: suppose that (\ref{g3}) does not hold; we could choose a set $\beta$ with ${\rm card}(\beta)\leq{\rm card}(\alpha)$ which is "better" than $\alpha$.

Let $H_1$ be the smallest integer such that $(J_\sigma)_{2^{-1}tc_\sigma}$ can be
covered by $H_1$ closed balls of radii $8^{-1}tc_\sigma$ which are centered in $(J_\sigma)_{2^{-1}tc_\sigma}$, we denote by $\gamma_1(\sigma)$ the centers of such $H_1$ closed balls. Let $H_2$ be the smallest integer such that $(J_\sigma)_{4^{-1}tc_\sigma}$ can be covered by $H_2$ closed balls of radii $8^{-1}tc_\sigma$. Let us denote by $\gamma_2(\sigma)$ the centers of such $H_2$ closed balls.  Then by (\ref{z17}), we have
\begin{eqnarray}\label{z18}
I_\sigma(\alpha)\geq I_\sigma(\alpha\cup\gamma_2(\sigma))\geq m_\sigma(\log c_\sigma+\hat{e}_{L_\sigma+H_2}(\nu_\sigma)).
\end{eqnarray}
Let $H_3$ be the smallest integer such that $J_\tau$ can be covered by $H_3$ closed balls of radii $8^{-1}tc_\tau$ which are centered in $J_\tau$ and  we denote by $\gamma_3(\tau)$ the centers of such $H_3$ closed balls. Let $L_0$ be the smallest integer such that (\ref{s8}) holds with $k_i=H_i,i=1,2,3$. Set $L_1:=L_0+H_1+H_3$.

Suppose that $L_\sigma={\rm card}(\alpha_\sigma)>L_1$. By (\ref{g4}), $\alpha_\sigma\cap\alpha_\omega=\emptyset$ for distinct words $\sigma,\omega\in\Lambda_j$. So, there is a $\tau\in\Lambda_j$ with $L_\tau=0$. Let $\gamma_4(\sigma)\in C_{L_\sigma-H_1-H_3}(\nu_\sigma)$. Set
\[
\beta:=(\alpha\setminus\alpha_\sigma)\cup\gamma_1(\sigma)\cup\gamma_3(\tau)\cup\gamma_4(\sigma).
\]
Then ${\rm card}(\alpha)\geq{\rm card}(\beta)$. The set $\gamma_1(\sigma)$ ensures that $J_\omega,\omega\in\Lambda_j\setminus\{\sigma,\tau\}$, are not affected unfavorably while we try to adjust the "optimal points" between $\alpha_\sigma$ and $\alpha_\tau$. In fact, by triangle inequality, we have, $d(x,\alpha_\sigma)\geq d(x,\gamma_1(\sigma))$ for $x\in J_\omega,\omega\in\Lambda_j\setminus\{\sigma,\tau\}$. It follows that
\begin{eqnarray}\label{z19}
I_\omega(\alpha)>I_\omega(\beta)\;\;{\rm for\;\;all}\;\;\omega\in\Lambda_j\setminus\{\sigma,\tau\}.
\end{eqnarray}
Thus, it suffices to estimate the following differences separately:
\begin{eqnarray*}
\Delta_1:=I_\sigma(\beta)-I_\sigma(\alpha);\;\;\Delta_2:=I_\tau(\alpha)-I_\tau(\beta).
\end{eqnarray*}
Using (\ref{z17}), (\ref{z8}) and (\ref{z18}), it is easy to show that $\Delta_1<\Delta_2$. This, together with (\ref{z19}),
implies that $I_\theta(\alpha)>I_\theta(\beta)$, contradicting the optimality of $\alpha$.
}\end{remark}

\begin{lemma}\label{characterization} There exists a constant $C_0$ such that for all large $j\in\mathbb{N}$,
\begin{eqnarray*}
\sum_{\tau\in\Lambda_j}m_\tau\log c_\tau+C_0\leq\hat{e}_{\psi_j}(\mu)\leq\sum_{\tau\in\Lambda_j}m_\tau\log c_\tau.
\end{eqnarray*}
\end{lemma}
\begin{proof}
Let $\alpha\in C_{\psi_j}(\mu)$ and let $\gamma_3(\tau)$ be as defined in Remark \ref{z20}. By (\ref{g3}),
${\rm card}(\alpha_\tau\cup\gamma_3(\tau))\leq L_1+H_3$ for every $\tau\in\Lambda_j$.
One can see that,
$$
d(x,\alpha)\geq d(x,\alpha_\tau\cup\gamma_3(\tau))\;\; {\rm for\; all}\;\;x\in J_\tau.
$$
Set $C_0:=B_{L_1+H_3}$. By (\ref{g6}) and Remark \ref{z7},
\begin{eqnarray*}
\hat{e}_{\psi_j}(\mu)\geq\sum_{\tau\in\Lambda_j}I_\tau(\alpha_\tau\cup\gamma_3(\tau))\geq\sum_{\tau\in\Lambda_j}m_\tau\log c_\tau+C_0.
\end{eqnarray*}
For each $\tau\in\Lambda_j$, let $b_\tau$ be an arbitrary point in $J_\tau$ and set $\gamma:=\{b_\tau\}_{\tau\in\Lambda_j}$. Then we have,
$\hat{e}_{\psi_j}(\mu)\leq I_\theta(\gamma)\leq\sum_{\tau\in\Lambda_j}m_\tau\log c_\tau$. The lemma follows.
\end{proof}

For $i\in G_1$, let $\mu_i$ denote the conditional probability measure $\mu(\cdot|J_i)$, namely, for every Borel set $B\subset\mathbb{R}^q$,
$\mu_i(B)=\mu(B\cap J_i)/\mu(J_i)$. We define
\begin{eqnarray*}
&&G_k(i):=\{\sigma\in G_k:\sigma_1=i\},\;\;k\geq 1;\;\;G^*(i):=\bigcup_{k\geq 1}G_k(i);\\
&&\Lambda_k(i):=\{\sigma\in G^*(i):p_{\sigma^-}\geq\underline{p}^k>p_\sigma\},\;\;\psi_k(i):={\rm card}(\Lambda_k(i)),\;\;k\geq 1.
\end{eqnarray*}

For $\sigma\in G^*(i)$, we have, $\mu_i(J_\sigma)=p_\sigma$. As we did for $\mu$, one can show that, there exists a constant $C_0(i)$ such that for all large $k\in\mathbb{N}$,
\begin{eqnarray}\label{mui}
\sum_{\sigma\in\Lambda_k(i)}p_\sigma\log c_\sigma+C_0(i)\leq\hat{e}_{\psi_k(i)}(\mu)\leq\sum_{\sigma\in\Lambda_k(i)}p_\sigma\log c_\sigma.
\end{eqnarray}

\section{Proof of Theorem \ref{mthm1}}
For the proof of Theorem \ref{mthm1}, we need to establish several lemmas. For $k,n\geq 1$, and $\sigma\in G_k$, let $\Gamma(\sigma,n):=\{\omega\in G_{k+n}:\sigma\prec\omega\}$; we define
\begin{eqnarray*}\label{s21}
&&\xi(i,n):=\sum_{\tau\in\Gamma(i,n)}p_\tau\log p_\tau;\;\;\lambda(i,n):=\sum_{\tau\in\Gamma(i,n)}p_\tau\log c_\tau,\;i\in G_1;\\
&&\Delta_n(i,j):=|\xi(i,n)-\xi(j,n)|,\;\;\widetilde{\Delta}_n(i,j):=|\lambda(i,n)-\lambda(j,n)|,\;\;i,j\in G_1.
\end{eqnarray*}

For $1\leq i,j\leq N$ with $(i,j)\notin G_2$, we have, $p_{ij}=c_{ij}=0$. In the following, we take the convention that $0\cdot\log 0:=0$, so that we may take the sums from $i=1$ to $N$,
instead of considering words in $G^*$. We always denote by $v=(v_i)_{i=1}^N$ the normalized positive left (row) eigenvector of $P$ with respect to the Perron-Frobenius eigenvalue $1$, when $P$ is irreducible. Let $u_0$ and $l_0$ denote the numerator and denominator in the definition of $s_0$ (see (\ref{s0})).

To study the asymptotics of the geometric mean errors, we will naturally need the following estimate which reflects some hereditary information of $\mu$. One may see \cite[(2.11)]{Zhu:13} for a comparison.

\begin{lemma}\label{lem0}
Assume that $P$ is irreducible. There exists a $C_1>0$ such that
\begin{eqnarray}\label{s1}
\sup_{n\geq 1}\max_{i,j\in G_1}\max\big\{\Delta_n(i,j),\widetilde{\Delta}_n(i,j)\big\}\leq C_1.
\end{eqnarray}
\end{lemma}
\begin{proof}
For $h\geq 1$ and $l,p\in G_1$, let $b_{lp}^{(h)}$ denote the $(l,p)$-entry of $P^h$.
For $h\geq 3$ and $l\in G_1$, we have, $\sum_{\tau\in G_{h-2}(i)}p_{\tau\ast l}=b_{il}^{(h-2)}$ (cf. \cite[(30)]{MW:88}). In addition, we have, $G_h(i)=\Gamma(i,h-1)$. One can see
\begin{eqnarray*}
\xi(i,h-1)&=&\sum_{\omega\in G_{h-1}(i)}\sum_{j=1}^Np_\omega p_{\omega_{h-1}j}\log p_\omega+
\sum_{\tau\in G_{h-2}(i)}\sum_{l=1}^N\sum_{j=1}^Np_{\tau\ast l}p_{lj}\log p_{lj}\\&=&\xi(i,h-2)+
\sum_{l=1}^N\sum_{j=1}^N b_{il}^{(h-2)}p_{lj}\log p_{lj}.
\end{eqnarray*}
Write $d_h(i):=\sum_{l=1}^N\sum_{j=1}^N b_{il}^{(h-2)}p_{lj}\log p_{lj},h\geq 3$. By induction, we have
\[
\xi(i,k-1)=\xi(i,1)+\sum_{h=3}^kd_h(i),\;\;k\geq 3.
\]
Set $w_{k,i}:=\xi(i,1)+(k-2)u_0$. Then we have
\begin{eqnarray}
\xi(i,k-1)=w_{k,i}+
\sum_{l=1}^N\sum_{j=1}^Np_{lj}\log p_{lj}\sum_{h=3}^k(b_{il}^{(h-2)}-v_l).\label{g20}
\end{eqnarray}
Similarly, set $z_{k,i}:=\lambda(i,1)+(k-2)l_0$. Then, for $k\geq 3$, we have
\begin{eqnarray}
\lambda(i,k-1)=z_{k,i}+\sum_{l=1}^N\sum_{j=1}^Np_{lj}\log c_{lj}\sum_{h=3}^k(b_{il}^{(h-2)}-v_l).
\end{eqnarray}
Let $u=(\chi_i)_{i=1}^N$ be the column vector with $\chi_i=1$ for all $1\leq i\leq N$. Then $u$ is a right eigenvector of $P$ with respect to $1$ and $\sum_{i=1}^N\chi_iv_i=1$. We have
\[
L:=uv=:(l_{ij})_{N\times N},\;\;l_{ij}=v_j,\;1\leq i,j\leq N.
\]
Applying \cite[Theorem 8.6.1]{HJ:85} with the above matrix $L$, there exists a constant $C(P)$ such that
for all $k\geq 3$,
\begin{equation*}
\frac{1}{k-2}\big|\sum_{h=3}^k(b_{pl}^{(h-2)}-v_l)\big|=\big|\frac{1}{k-2}\sum_{h=3}^kb_{pl}^{(h-2)}-v_l\big|<\frac{C(P)}{k-2},\;\;p,l\in G_1.
\end{equation*}
This, together with (\ref{g20}), yields
\begin{eqnarray}\label{suc1}
\frac{1}{k-2}|\xi(i,k-1)-w_{k,i}|\leq\frac{C(P)}{k-2}\sum_{l=1}^N\sum_{j=1}^N|p_{lj}\log p_{lj}|=:\frac{\delta_0}{k-2};\;k\geq 3.
\end{eqnarray}
Hence, $|\xi(i,k-1)-w_{k,i}|\leq\delta_0$ for all $k\geq 3$. Note that, the above argument is true for all $i\in G_1$. Set
$\delta_1:=\max_{i,j\in G_1}|\xi(i,1)-\xi(j,1)|$. Then for $n=k-1$,
\begin{equation*}
\Delta_n(i,j)\leq|\xi(i,n)-w_{k,i}|+|w_{k,i}-w_{k,j}|+|\xi(j,n)-w_{k,j}|\leq 2\delta_0+\delta_1=:\delta_2.
\end{equation*}
Analogously, for some constant $\delta_3>0$, we have, $\widetilde{\Delta}_n(i,j)\leq\delta_3$ for $i,j\in G_1$ and $h\geq 1$.
Thus, the lemma follows by setting $C_1:=\max\{\delta_2,\delta_3\}$.
\end{proof}

The following two number sequences $(t_k)_{k=1}^\infty$ and $(s_k)_{k=1}^\infty$ are closely connected with the asymptotic geometric mean errors:
\begin{eqnarray}\label{g10}
t_k:=\frac{\sum_{\sigma\in\Lambda_k}m_\sigma\log m_\sigma}{\sum_{\sigma\in\Lambda_k}m_\sigma\log c_\sigma};\;\;
s_k:=\frac{\sum_{\sigma\in G_k}m_\sigma\log m_\sigma}{\sum_{\sigma\in G_k}m_\sigma\log c_\sigma},\;k\geq 1.
\end{eqnarray}
Let $u_k,l_k$ denote the numerator and denominator in the definition of $s_k$. Then
\begin{eqnarray}\label{z15}
u_1=\sum_{i=1}^Nq_i\log q_i,\;\;u_2=\sum_{i=1}^N\sum_{j=1}^Nq_ip_{ij}\log(q_ip_{ij})=u_1+\sum_{i=1}^Nq_i\xi(i,1).
\end{eqnarray}
\begin{lemma}\label{g16}
Assume that $P$ is irreducible. There exists a constant $C_2$ such that for all large $k\in\mathbb{N}$, we have
$|s_k-s_0|\leq C_2 k^{-1}$.
\end{lemma}
\begin{proof}
Let $w_{k,i},z_{k,i}$ be as defined in the proof of Lemma \ref{lem0}. We write
\begin{eqnarray*}
x_k:=\sum_{i=1}^Nq_i(\xi(i,k-1)-w_{k,i}),\;y_k:=\sum_{i=1}^Nq_i(\lambda(i,k-1)-z_{k,i}),\;k\geq 3.
\end{eqnarray*}
For $k\geq 3$, by (\ref{g20}), (\ref{z15}) and the definitions of $u_k$ and $\xi(i,k-1)$, we deduce
\begin{eqnarray}
u_k&=&\sum_{i=1}^N\sum_{\omega\in G_k(i)}q_i p_\omega\log q_i+\sum_{i=1}^N\sum_{\omega\in G_k(i)}q_i p_\omega\log p_\omega\nonumber\\
&=&u_1+\sum_{i=1}^Nq_i\xi(i,k-1)\nonumber\\
&=&u_1+\sum_{i=1}^Nq_i\big(\xi(i,1)+(k-2)u_0+\xi(i,k-1)-w_{k,i}\big)\nonumber\\
&=&u_2+(k-2)u_0+x_k.\label{s23}
\end{eqnarray}
Note that $l_1=0$. By replacing $\log p_{ij}$ in (\ref{s23}) with $\log c_{ij}$, we have
\begin{eqnarray}\label{s25}
l_k=l_2+(k-2)l_0+y_k.
\end{eqnarray}
By (\ref{suc1}), we have, $|\xi(i,k-1)-w_{k,i}|\leq\delta_0$ for all $i\in G_1$. Thus,
\begin{eqnarray}\label{s26}
|x_k|\leq\sum_{i=1}^Nq_i|\xi(i,k-1)-w_{k,i}|\leq\delta_0 \;\;{\rm and}\;\;|y_k|\leq\delta_3.
\end{eqnarray}
Set $s_2:=u_2+l_2$ and $A_6:=|s_2(l_0+u_0)|$. by (\ref{s23})-(\ref{s26}), for large $k$, we deduce
\begin{eqnarray*}
|s_k-s_0|&=&\bigg|\frac{u_2+(k-2)u_0+x_k}
{l_2+(k-2)l_0+y_k}-\frac{u_0}{l_0}\bigg|\\&\leq&
\frac{A_6+|x_kl_0-y_ku_0|}{|l_0(l_2+(k-2)l_0+y_k)|}\\
&\leq&\frac{2A_6+2(\delta_0+\delta_3)(|u_0|+|l_0|)}{(k-2)l_0^2}\\&\leq&\frac{4A_6+4(\delta_0+\delta_3)(|u_0|+|l_0|)}{kl_0^2}=:C_2k^{-1}.
\end{eqnarray*}
This completes the proof of the lemma.
\end{proof}

\begin{remark}{\rm
If $(q_i)_{i=1}^N$ agrees with $v$, then $\sum_{i=1}^Nq_i b_{il}^{(h-2)}=v_l$ since $v P^k=v$ for all $k\geq 1$. Hence, $x_k=0$ and (\ref{s23}) becomes:
$u_k=u_1+(k-2)u_0$. This was calculated in \cite[Theorem 4.27]{PW:82}.
}\end{remark}

Next we establish a connection between $(t_j)_{j=1}^\infty$ and $(s_k)_{k=1}^\infty$. We have
\begin{lemma}\label{lem1}
Assume that $P$ is irreducible. There exist a constant $C_3$ and two integers $k_j^{(i)}\in[k_{1j},k_{2j}],i=1,2$ such that
\begin{equation*}
s_{k_j^{(1)}}-C_3j^{-1}\leq t_j\leq s_{k_j^{(2)}}+C_3j^{-1}.
\end{equation*}
\end{lemma}
\begin{proof}
For $k\geq 1$ and $\sigma\in G_k$, we have
\begin{eqnarray}\label{z5}
&&\sum_{\tau\in\Gamma(\sigma,k_{2j}-|\sigma|)}m_\tau\log m_\tau=
\sum_{\tau\in\Gamma(\sigma,k_{2j}-|\sigma|)}m_\tau\big(\log m_\sigma+\log\frac{m_\tau}{m_\sigma}\big)\nonumber\\
&&\;\;\;\;=m_\sigma\log m_\sigma+m_\sigma\sum_{\tau\in\Gamma(\sigma,k_{2j}-|\sigma|)}\frac{m_\tau}{m_\sigma}\log\frac{m_\tau}{m_\sigma}
\nonumber\\&&\;\;\;\;=m_\sigma\log m_\sigma+m_\sigma\xi(\sigma_{|\sigma|},k_{2j}-|\sigma|).
\end{eqnarray}
By (\ref{z5}) and Lemma \ref{lem0}, for every $k\in[k_{1j},k_{2j}]$ and $\omega\in G_k$, we have
\begin{eqnarray}\label{z16}
&&u_{k_{2j}}-u_k=\sum_{\sigma\in G_k}\sum_{\tau\in\Gamma(\sigma,k_{2j}-k)}m_\tau\log m_\tau-u_k=\sum_{\sigma\in G_k}m_\sigma\xi(\sigma_k,k_{2j}-k)\nonumber\\&&\leq
\sum_{\sigma\in G_k}m_\sigma(\xi(\omega_k,k_{2j}-k)+C_1)=
\xi(\omega_k,k_{2j}-k)+C_1.
\end{eqnarray}
Similarly, we have that $u_{k_{2j}}-u_k\geq\xi(\omega_k,k_{2j}-k)-C_1$. We define
$$
\zeta(\sigma):=m_\sigma(\log m_\sigma-u_{|\sigma|}),\;\;\sigma\in G^*\setminus\{\theta\}.
$$
Then, by (\ref{z5}) and (\ref{z16}), we deduce
\begin{eqnarray*}
&&\sum_{\tau\in\Gamma(\sigma,k_{2j}-|\sigma|)}\zeta(\tau)=\sum_{\tau\in\Gamma(\sigma,k_{2j}-|\sigma|)}m_\tau(\log m_\tau-u_{k_{2j}})\\&&\;\;=m_\sigma\log m_\sigma+m_\sigma\xi(\sigma_{|\sigma|},k_{2j}-|\sigma|)-m_\sigma u_{k_{2j}}\\
&&\left\{\begin{array}{ll}\leq m_\sigma\log m_\sigma+m_\sigma(u_{k_{2j}}-u_{|\sigma|}+C_1)-m_\sigma u_{k_{2j}}=\zeta(\sigma)+m_\sigma C_1.\\
\geq m_\sigma\log m_\sigma+m_\sigma(u_{k_{2j}}-u_{|\sigma|}-C_1)-m_\sigma u_{k_{2j}}=\zeta(\sigma)-m_\sigma C_1\end{array}\right.
\end{eqnarray*}
This is equivalent to
\[
\sum_{\tau\in\Gamma(\sigma,k_{2j}-|\sigma|}\zeta(\tau)-m_\sigma C_1\leq\zeta(\sigma)\leq\sum_{\tau\in\Gamma(\sigma,k_{2j}-|\sigma|}\zeta(\tau)+m_\sigma C_1.
\]
Note that $\sum_{\tau\in G_{k_{2j}}}\zeta(\tau)=0$. We further deduce
\begin{eqnarray*}
&&\sum_{\sigma\in\Lambda_j}m_\sigma(\log m_\sigma-u_{|\sigma|})=\sum_{\sigma\in\Lambda_j}\zeta(\sigma)\\&&
\left\{\begin{array}{ll}\leq \sum_{\sigma\in\Lambda_j}\big(\sum_{\tau\in\Gamma(\sigma,k_{2j}-|\sigma|}\zeta(\tau)+m_\sigma C_1\big)=C_1.\\
\geq\sum_{\sigma\in\Lambda_j}\big(\sum_{\tau\in\Gamma(\sigma,k_{2j}-|\sigma|}\zeta(\tau)-m_\sigma C_1\big)=-C_1.\end{array}\right.
\end{eqnarray*}
As an immediate consequence, we have
\begin{eqnarray}\label{z2}
\sum_{\sigma\in\Lambda_j}m_\sigma u_{|\sigma|}-C_1\leq\sum_{\sigma\in\Lambda_j}m_\sigma\log m_\sigma\leq\sum_{\sigma\in\Lambda_j}m_\sigma u_{|\sigma|}+C_1.
\end{eqnarray}
Similarly, one can show that
\begin{eqnarray}\label{z3}
\sum_{\sigma\in\Lambda_j}m_\sigma l_{|\sigma|}-C_1\leq\sum_{\sigma\in\Lambda_j}m_\sigma\log c_\sigma\leq\sum_{\sigma\in\Lambda_j}m_\sigma l_{|\sigma|}+C_1.
\end{eqnarray}
Thus, for large $j$, by (\ref{z2}) and (\ref{z3}), we have
\begin{eqnarray}\label{z10}
\frac{\sum_{\sigma\in\Lambda_j}m_\sigma u_{|\sigma|}+C_1}{\sum_{\sigma\in\Lambda_j}m_\sigma l_{|\sigma|}-C_1}\leq\frac{\sum_{\sigma\in\Lambda_j}m_\sigma\log m_\sigma}{\sum_{\sigma\in\Lambda_j}m_\sigma\log c_\sigma}\leq\frac{\sum_{\sigma\in\Lambda_j}m_\sigma u_{|\sigma|}-C_1}{\sum_{\sigma\in\Lambda_j}m_\sigma l_{|\sigma|}+C_1}.
\end{eqnarray}
Let $A_7:=2^{-1}|l_0|$. Note that $l_k,k\geq 2$, are all negative. By (\ref{s25}) and (\ref{s26}),
\begin{eqnarray}\label{z14}
|l_k|\geq (k-2)|l_0|-\delta_0\geq2^{-1}k|l_0|=A_7k,\;\;k\geq4+\delta_0|l_0|^{-1}.
\end{eqnarray}
This, together with the definition of $k_{1j}$, implies
\begin{eqnarray}\label{z4}
\bigg|\sum_{\sigma\in\Lambda_j}m_\sigma l_{|\sigma|}-C_1\bigg|\geq\sum_{\sigma\in\Lambda_j}m_\sigma |l_{|\sigma|}|\geq A_7k_{1j}.
\end{eqnarray}
By Lemma \ref{g16}, we have, $s_k\leq 2s_0$ for all large $k$. Hence,
\[
\frac{\sum_{\sigma\in\Lambda_j}m_\sigma u_{|\sigma|}}{\sum_{\sigma\in\Lambda_j}m_\sigma l_{|\sigma|}}\leq\max_{k_{1j}\leq k\leq k_{2j}}s_k\leq 2s_0.
\]
Combining this with (\ref{z4}), we have
\begin{eqnarray}\label{z12}
&&\bigg|\frac{\sum_{\sigma\in\Lambda_j}m_\sigma u_{|\sigma|}+C_1}{\sum_{\sigma\in\Lambda_j}m_\sigma l_{|\sigma|}-C_1}-\frac{\sum_{\sigma\in\Lambda_j}m_\sigma u_{|\sigma|}}{\sum_{\sigma\in\Lambda_j}m_\sigma l_{|\sigma|}}\bigg|\nonumber\\&&=\frac{C_1\big(|\sum_{\sigma\in\Lambda_j}m_\sigma l_{|\sigma|}|+|\sum_{\sigma\in\Lambda_j}m_\sigma u_{|\sigma|}|\big)}{|\big(\sum_{\sigma\in\Lambda_j}m_\sigma l_{|\sigma|}-C_1\big)\sum_{\sigma\in\Lambda_j}m_\sigma l_{|\sigma|}|}\leq\frac{C_1(1+2s_0)}{A_7k_{1j}}.
\end{eqnarray}
For large $j$, we have, $C_1\leq 2^{-1}|\sum_{\sigma\in\Lambda_j}m_\sigma l_{|\sigma|}|$. Hence, we similarly get
\begin{eqnarray}\label{z11}
&&\bigg|\frac{\sum_{\sigma\in\Lambda_j}m_\sigma u_{|\sigma|}-C_1}{\sum_{\sigma\in\Lambda_j}m_\sigma l_{|\sigma|}+C_1}-\frac{\sum_{\sigma\in\Lambda_j}m_\sigma u_{|\sigma|}}{\sum_{\sigma\in\Lambda_j}m_\sigma l_{|\sigma|}}\bigg|\leq\frac{2C_1(1+2s_0)}{A_7k_{1j}}.
\end{eqnarray}
Set $A_8:=2C_1(1+2s_0)A_7^{-1}$. By (\ref{z10}), (\ref{z12}) and (\ref{z11}), we deduce
\begin{eqnarray*}\label{g15}
\frac{\sum_{\sigma\in\Lambda_j}m_\sigma u_{|\sigma|}}{\sum_{\sigma\in\Lambda_j}m_\sigma l_{|\sigma|}}-\frac{A_8}{k_{1j}}
\leq\frac{\sum_{\sigma\in\Lambda_j}m_\sigma\log m_\sigma}{\sum_{\sigma\in\Lambda_j}m_\sigma\log c_\sigma}
\leq\frac{\sum_{\sigma\in\Lambda_j}m_\sigma u_{|\sigma|}}{\sum_{\sigma\in\Lambda_j}m_\sigma l_{|\sigma|}}+\frac{A_8}{k_{1j}}.
\end{eqnarray*}
Now one can see that, there exist some $k_j^{(i)}\in[k_{1j},k_{2j}],i=1,2$, such that
\begin{eqnarray*}
s_{k_j^{(1)}}=\min_{k_{1j}\leq k\leq k_{2j}}\frac{u_k}{l_k}\leq
\frac{\sum_{\sigma\in\Lambda_j}m_\sigma u_{|\sigma|}}{\sum_{\sigma\in\Lambda_j}m_\sigma l_{|\sigma|}}
\leq\max_{k_{1j}\leq k\leq k_{2j}}\frac{u_k}{l_k}=s_{k_j^{(2)}}.
\end{eqnarray*}
Thus, in view of (\ref{g7}), the lemma follows by setting $C_3:=A_8/A_1$.
\end{proof}

With the above analysis, we obtain the convergence order of $(t_j)_{j=1}^\infty$:
\begin{lemma}\label{z6}
Assume that $P$ is irreducible. There exists a constant $C_4$ such that $|t_j-s_0|<C_4 j^{-1}$ for all large $j$.
\end{lemma}
\begin{proof}
By Lemmas \ref{g16}, \ref{lem1} and Lemma \ref{g7}, we have
\begin{eqnarray*}
t_j-s_0\left\{\begin{array}{ll}\leq s_{k_j^{(2)}}-s_0+C_3j^{-1}\leq (C_2A_1^{-1}+C_3)j^{-1}\\
\geq s_{k_j^{(1)}}-s_0-C_3j^{-1}\geq-(C_2A_1^{-1}+C_3)j^{-1}\end{array}\right..
\end{eqnarray*}
Hence, the lemma follows by setting $C_4:=C_2A_1^{-1}+C_3$.
\end{proof}

Now we are able to give the proof of Theorem \ref{mthm1}. For a Borel probability measure $\nu$ on $\mathbb{R}^q$ and every $n\geq 1$, we write
$$
Q_n(\nu,a):=a^{-1}\log n+\hat{e}_n(\nu),\;\;a>0.
$$

\emph{Proof of Theorem \ref{mthm1}}
By (\ref{lambdaj}), we easily see
\begin{eqnarray}\label{z9}
\underline{p}^{j+1}\leq p_\sigma<\underline{p}^j;\;\;\overline{q}^{-1}\underline{p}^{-j}\leq\psi_j\leq\underline{q}^{-1}\underline{p}^{-(j+1)}.
\end{eqnarray}
Since $t_j\to s_0$, we have, $2^{-1}s_0\leq t_j\leq 2s_0$ for all large $j$. By (\ref{z9}),
\begin{eqnarray*}
&&\sum_{\sigma\in\Lambda_j}m_\sigma\log c_\sigma=t_j^{-1}\sum_{\sigma\in\Lambda_j}m_\sigma\log m_\sigma=t_j^{-1}\sum_{\sigma\in\Lambda_j}m_\sigma\log (q_{\sigma_1}p_\sigma)\\&&=t_j^{-1}\sum_{\sigma\in\Lambda_j}m_\sigma\log q_{\sigma_1}+t_j^{-1}\sum_{\sigma\in\Lambda_j}m_\sigma\log p_\sigma\left\{\begin{array}{ll}\leq (2s_0)^{-1}\log\overline{q}+t_j^{-1}\log\underline{p}^j\\
\geq 2s_0^{-1}\log\underline{q}+t_j^{-1}\log\underline{p}^{j+1}\end{array}\right..
\end{eqnarray*}
This, together with Lemma \ref{characterization}, yields
\begin{eqnarray*}\label{g17}
 2s_0^{-1}\log\underline{q}+t_j^{-1}\log\underline{p}^{j+1}+C_0\leq\hat{e}_{\psi_j}(\mu)\leq t_j^{-1}\log\underline{p}^j+(2s_0)^{-1}\log\overline{q}.
\end{eqnarray*}
Thus, by Lemma \ref{lem1}, (\ref{z9}) and the fact that $2^{-1}s_0\leq t_j\leq 2s_0$, we deduce
\begin{eqnarray*}
&&Q_{\psi_j}(\mu,s_0)=s_0^{-1}\log\psi_j+\hat{e}_{\psi_j}(\mu)\\&&\;\;\left\{\begin{array}{ll}\leq (s_0^{-1}-t_j^{-1})\log\underline{p}^{-j}-s_0^{-1}\log(\underline{q}\underline{p})+(2s_0)^{-1}\log\overline{q}\\
\geq (s_0^{-1}-t_j^{-1})\log\underline{p}^{-j}+2s_0^{-1}\log\underline{q}+2s_0^{-1}\log\underline{p}+C_0\end{array}\right..
\end{eqnarray*}
Finally, using Lemma \ref{z6}, we obtain
\begin{eqnarray*}
Q_{\psi_j}(\mu,s_0)\left\{\begin{array}{ll}\leq 2C_4s_0^{-2}\log\underline{p}^{-1}-s_0^{-1}\log(\underline{q}\underline{p})+(2s_0)^{-1}\log\overline{q}\\
\geq -2C_4s_0^{-2}\log\underline{p}^{-1}+2s_0^{-1}\log\underline{q}+2s_0^{-1}\log\underline{p}+C_0\end{array}\right..
\end{eqnarray*}
Hence, $0<\underline{P}_0^{s_0}(\mu)\leq\overline{P}_0^{s_0}(\mu)<\infty$.
By Lemma \ref{g7}, the theorem follows.

\vspace{0.2cm}
In the following, we study the asymptotic geometric mean error for the conditional probability measures $\mu_i$. For every $i\in G_1$, we write
\begin{eqnarray*}
t_k(i):=\frac{\sum_{\sigma\in\Lambda_k(i)}p_\sigma\log p_\sigma}{\sum_{\sigma\in\Lambda_k(i)}p_\sigma\log c_\sigma};\;\;
s_k(i):=\frac{\sum_{\sigma\in G_k(i)}p_\sigma\log p_\sigma}{\sum_{\sigma\in G_k(i)}p_\sigma\log c_\sigma},\;\;k\geq 2.
\end{eqnarray*}
Let us denote by $u_k(i),l_k(i)$ the numerator and denominator in the definition of $s_k(i)$. We have the following estimate for the convergence order of $(t_j(i))_{j=1}^\infty$.
\begin{lemma}\label{g24}
Assume that $P$ is irreducible. There exists a constant $C_5$ such that $|t_j(i)-s_0|<C_5 j^{-1}$ for every $i\in G_1$ and all large $j$.
\end{lemma}
\begin{proof}
Fix an arbitrary $i\in G_1$. By Lemma \ref{lem0}, for every pair $l,h\in G_1$,
\begin{eqnarray*}
u_k(h)-C_1\leq u_k(l)=\xi(l,k-1)\leq\xi(h,k-1)+C_1= u_k(h)+C_1.
\end{eqnarray*}
Thus, for the above $i$ and $k\geq 2$, we have
\begin{eqnarray*}
Nu_k(i)-NC_1\leq u_k=\sum_{j=1}^Nu_k(j)\leq Nu_k(i)+NC_1.
\end{eqnarray*}
Set $a_k:=u_k(i)-N^{-1}u_k$ and $b_k:=l_k(i)-N^{-1}l_k$. Then $|a_k|,|b_k|\leq C_1$. Note that $|l_k|\to\infty$
as $k\to\infty$. Hence, for large $k$, we have,
$$
|N^{-1}l_k+b_k|\geq 2^{-1}|N^{-1}l_k|\;\;{\rm and}\;\;s_k\leq 2s_0.
$$
Using these facts and (\ref{z14}), we deduce
\begin{eqnarray}\label{g23}
&&|s_k(i)-s_k|=\bigg|\frac{u_k(i)}{l_k(i)}-\frac{u_k}{l_k}\bigg|=\bigg|\frac{N^{-1}u_k+a_k}{N^{-1}l_k+b_k}-\frac{u_k}{l_k}\bigg|
\nonumber\\&&\leq\frac{C_1(|u_k|+|l_k|)}{|l_k(N^{-1}l_k+b_k)|}\leq\frac{C_1}{|N^{-1}l_k+b_k|}+
\frac{C_1|u_k|}{|l_k(N^{-1}l_k+b_k)|}\nonumber\\&&\leq\frac{2N_1C_1}{|l_k|}+
\frac{2N_1C_1s_k}{|l_k|}\leq\frac{2N_1C_1(1+2s_0)}{k A_7}=:A_9k^{-1}.
\end{eqnarray}
Let $k_{1j}(i):=\min_{\sigma\in\Lambda_j(i)}|\sigma|$ and $k_{2j}(i):=\max_{\sigma\in\Lambda_j(i)}|\sigma|$. By Lemma \ref{g7},
\[
A_1j\leq k_{1j}\leq k_{1j}(i)\leq k_{2j}(i)\leq k_{2j}\leq A_2j.
\]
Along the line in the proof of Lemma \ref{lem1}, one can show, there exist a constant $C_3(i)$ and two integers $k_j^{(h)}\in[k_{1j}(i),k_{2j}(i)],h=1,2$, such that
\begin{equation*}
s_{k_j^{(1)}}(i)-C_3(i)j^{-1}\leq t_j(i)\leq s_{k_j^{(2)}}(i)+C_3(i)j^{-1}.
\end{equation*}
Combining this and (\ref{g23}), we have
\begin{eqnarray*}
t_j(i)-s_0&\leq& s_{k_j^{(2)}}(i)+C_3(i)j^{-1}-s_0\leq A_9k_j^{(2)}(i)^{-1}+C_3(i)j^{-1}\\&\leq&
 A_9A_1^{-1}j^{-1}+C_3(i)j^{-1}=(A_9A_1^{-1}+C_3(i))j^{-1}.
\end{eqnarray*}
Similarly, one can show that $t_j(i)-s_0\geq -(A_9A_1^{-1}+C_3(i))j^{-1}$. Thus, the lemma follows by setting $C_5:=A_9A_1^{-1}+\max_{i\in G_1}C_3(i)$.
\end{proof}
\begin{proposition}\label{g28}
Assume that $P$ is irreducible. Then, we have
$$
0<\underline{Q}_0^{s_0}(\mu_i)\leq\overline{Q}_0^{s_0}(\mu_i)<\infty,\;\;i\in G_1.
$$
\end{proposition}
\begin{proof}
It suffices to follow the proof of Theorem \ref{mthm1} by using (\ref{mui}) and Lemma \ref{g24}. We omit the details.
\end{proof}

Next, we show that, when the transition matrix $P$ is reducible, $D_0(\mu)$ is dependent on the initial probability vector. For this, we need the following observation which is a consequence of the arguments in \cite[Example 4.1]{GL:04}.
\begin{proposition}\label{g27}
Let $\nu_i,1\leq i\leq N$, be Borel probability measures on $\mathbb{R}^q$ of compact support such that, for all $\epsilon>0$, we have
$$
\max_{1\leq i\leq N}\sup_{x\in\mathbb{R}^q}\nu_i(B(x,\epsilon))\leq d_1\epsilon^{d_2}.
$$
Assume that $D_0(\nu_i)=t_i>0,i\in G_1$. Let $(q_i)_{i=1}^N$ be a probability vector with $q_i>0$ for all $i\in G_1$.
Then for $\nu=\sum_{i=1}^Nq_i\nu_i$, we have $D_0(\nu)=t_0$, where
\[
t_0=\frac{t_1t_2\cdots t_N}{q_1t_2\cdots t_N+\ldots+q_N t_1\cdots t_{N-1}}.
\]
Moreover, we have, $\underline{Q}_0^{t_0}(\nu)>0$ if $\underline{Q}_0^{t_i}(\nu_i)>0$ for all $1\leq i\leq N$; and $\overline{Q}_0^{t_0}(\nu)<\infty$ if $\overline{Q}_0^{t_i}(\nu_i)<\infty$ for all $1\leq i\leq N$.
\end{proposition}
\begin{proof}
We denote by $[x]$ the largest integer not greater than $x\in\mathbb{R}$. By the arguments in \cite[Example 4.1]{GL:04}, we have
\begin{eqnarray}\label{g25}
\hat{e}_n(\nu)\left\{\begin{array}{ll}\geq q_1\hat{e}_n(\nu_1)+\ldots+q_N\hat{e}_n(\nu_N)\\
\leq q_1\hat{e}_{[\frac{n}{N}]}(\nu_1)+\ldots+q_N\hat{e}_{[\frac{n}{N}]}(\nu_N)\end{array}\right..
\end{eqnarray}
By (\ref{quanerrordef}) and (\ref{g25}),we easily see that $D_0(\nu)=t_0$. Furthermore, by (\ref{g25}),
\begin{eqnarray}\label{z21}
Q_n(\nu,t_0)\geq t_0^{-1}\log n+\sum_{i=1}^Nq_i\hat{e}_{n}(\nu_i))=\sum_{i=1}^Nq_i Q_n(\nu_i,t_i).
\end{eqnarray}
Set $C_5:=(q_1t_1^{-1}+\ldots+q_Nt_N^{-1})\log (2N)$. Then, for all $n\geq 2N$, we have
\begin{eqnarray}\label{z22}
Q_n(\nu,t_0)\leq t_0^{-1}\log n+\sum_{i=1}^Nq_i\hat{e}_{[\frac{n}{N}]}(\nu_i)=
\sum_{i=1}^Nq_i Q_{[\frac{n}{N}]}(\nu_i,t_i)+C_5.
\end{eqnarray}
By (\ref{z21}) and (\ref{z22}), we conclude
\begin{eqnarray*}
\prod_{i=1}^N(\underline{Q}_0^{t_i}(\nu_i))^{q_i}\leq\underline{Q}_0^{t_0}(\nu)\leq
\overline{Q}_0^{t_0}(\nu)\leq e^{C_5}\prod_{i=1}^N(\overline{Q}_0^{t_i}(\nu_i))^{q_i}.
\end{eqnarray*}
This implies the second assertion of the proposition.
\end{proof}

\begin{example}{\rm
Let $Q_1=(q_{ij})_{i,j=1}^2, Q_2=(t_{ij})_{i,j=3}^4$ be positive row-stochastic matrices ($q_{ij}>0,1\leq i,j\leq 2$ and $t_{ij}>0,3\leq i,j\leq 4$). Let $P$ denote the block diagonal matrix ${\rm diag}(Q_1,Q_2)$. Then $P$ is reducible. Let $(c_{ij})_{4\times 4}$ be given and assume that (\ref{g4}) holds. Let $\mu$ be the Markov-type measure associated with $P$ and initial probability vector $(q_i)_{i=1}^4$. Clearly, $Q_1,Q_2$ are both irreducible. Let $(v_i^{(h)})_{i=1}^2$ be the normalized positive left eigenvector of $Q_h$ for $h=1,2$. Write
\begin{eqnarray*}
t_1:=\frac{\sum_{i=1}^2 v_i^{(1)}\sum_{j=1}^2q_{ij}\log q_{ij}}{\sum_{i=1}^2 v_i^{(1)}\sum_{j=1}^2q_{ij}\log c_{ij}};\;\;
t_2:=\frac{\sum_{i=3}^4 v_i^{(2)}\sum_{j=3}^4t_{ij}\log t_{ij}}{\sum_{i=3}^4 v_i^{(2)}\sum_{j=3}^4t_{ij}\log c_{ij}}.
\end{eqnarray*}
By Proposition \ref{g28}, we have, $0<\underline{Q}_0^{t_1}(\mu_i)\leq\overline{Q}_0^{t_1}(\mu_i)<\infty$ for $i=1,2$; and for $i=3,4$, $0<\underline{Q}_0^{t_2}(\mu_i)\leq\overline{Q}_0^{t_2}(\mu_i)<\infty$. Set
$$
t_0:=\frac{t_1t_2}{(q_1+q_2)t_2+(q_3+q_4)t_1}.
$$
Then by Proposition \ref{g27},  we have, $0<\underline{Q}_0^{t_0}(\mu)\leq\overline{Q}_0^{t_0}(\mu)<\infty$.
In this example, $D_0(\mu)$ depends on the initial probability vector provided that $t_1\neq t_2$.
}\end{example}

\noindent{\bf Acknowledgement} The author thanks the referee for some helpful comments.

\end{document}